\documentclass[12pt,a4paper]{article}
\usepackage{epsfig}
\usepackage{color} 
\usepackage{latexsym}
\usepackage{amsfonts}
\usepackage{amsmath, amssymb, amsthm}
\usepackage{graphicx}

\newtheorem{theorem}{Theorem}

\newtheorem{corollary}[theorem]{Corollary}

\newtheorem*{theorem1*}{Theorem 1}
\newtheorem*{theorem2*}{Theorem 2}
\newtheorem*{conjecture1*}{Conjecture 1}
\newtheorem*{conjecture2*}{Conjecture 2}

\begin{document}

\title{\large{\textbf{CRITICAL AND FLOW-CRITICAL SNARKS COINCIDE}}}

\author{
Edita M\'a\v cajov\' a and Martin \v{S}koviera\\[3mm]
Department of Computer Science\\
Faculty of Mathematics, Physics and Informatics\\
Comenius University\\
842 48 Bratislava, Slovakia\\[2mm]
{\small\tt macajova@dcs.fmph.uniba.sk}\\[-1mm]
{\small\tt skoviera@dcs.fmph.uniba.sk}}

\date{\today}

\maketitle

\begin{abstract}
Over the past twenty years, critical and bicritical snarks have
been appearing in the literature in various forms and in
different contexts. Two main variants of criticality of snarks
have been studied: criticality with respect to the
non-existence of a $3$-edge-colouring and criticality with
respect to the non-existence of a nowhere-zero $4$-flow. In
this paper we show that these two kinds of criticality
coincide, thereby completing previous partial results of de
Freitas et al. [Electron. Notes Discrete Math. 50 (2015),
199--204] and Fiol et al. [ arXiv:1702.07156v1 (2017)].
\end{abstract}

\section{Introduction}

A snark is a connected cubic graph whose edges cannot
be properly coloured with three colours; equivalently, it is a
connected cubic graph that has no nowhere-zero $4$-flow. This
definition follows Cameron et al. \cite{CCW}, Nedela and \v
Skoviera \cite{NS}, \v S\' amal \cite{S-full}, Steffen
\cite{Steffen:class}, and others, rather than the traditional
more restrictive definition that excludes small
cycle-separating edge-cuts and short circuits in order to avoid
trivial cases. As suggested by several authors, the idea of
nontriviality of snarks is rather subtle and seems to be best
captured by various reductions and decompositions of snarks
\cite{CCW, ChS, NS, Steffen:class}. The concept of a snark
reduction is, in turn, closely related to that of criticality
of a snark, which naturally takes one of two forms: criticality
with respect to the non-existence of a $3$-edge-colouring
\cite{BS, ChS, NS} and criticality with respect to the
non-existence of a nowhere-zero $4$-flow \cite{CdSM, dFdSL,
dSL, dSPL}. The purpose of the present paper is to show that
these two types of criticality coincide. Although the discussed
relationship is not complicated, it has been generally
overlooked and, so far, the two types of criticality have been
considered separately, see, for example \v S\' amal
\cite[Section 3.1]{S-full}. Even the most recent survey paper
of Fiol et al. brings only a partial result in this direction
\cite[Theorem~4.5]{FMS}.

The concept of a critical snark first appeared in 1996 in the
work of Nedela and \v Skoviera \cite{NS} within the context of
snark reductions. According to their definition, a snark is
critical if the removal of any two adjacent vertices produces a
$3$-edge-colourable graph, and bicritical if the removal of any
two distinct vertices produces a $3$-edge-colourable graph.

Explicit occurrence of flow-critical snarks in the literature
is of much later date. It first appears in a 2008 paper of da
Silva and Lucchesi \cite{dSL} investigating graphs critical
with respect to the existence of a nowhere-zero $k$-flow for an
arbitrary integer $k\ge 2$. They defined a graph to be
$k$-edge-critical if it does not admit a nowhere-zero $k$-flow
but the graph obtained by the contraction of any edge does.
They further defined a graph to be $k$-vertex-critical if it
does not admit a nowhere-zero $k$-flow but the graph obtained
by the identification of any two distinct vertices does. If we
take into account the fact that contracting an edge has the
same effect on the existence of a nowhere-zero $k$-flow as
identifying its end-vertices, the later two definitions of da
Silva and Lucchesi \cite{dSL}, with $k=4$, can be viewed as
natural counterparts of critical and bicritical snarks of
Nedela and \v Skoviera. Nevertheless, our main result shows
that for snarks flow-criticality does not bring anything
substantially new.

\begin{theorem}
A snark is $4$-edge-critical if and on if it is critical. A
snark is $4$-vertex-critical if and only if it is bicritical.
\end{theorem}

There have been a number of papers following either of the two
approaches to the criticality; see for example \cite{ChS,
GS,Steffen:strict, Steffen:bicrit} and \cite{CdSM, dSPL,
dFdSL}, respectively. In several other works, critical snarks
have emerged in forms different from those explained above, yet
in all the cases the definitions turn out to be equivalent to
one of those given above. For example, DeVos et al.~\cite{dVNR}
and more recently \v S\' amal \cite{S-full} define a snark to
be critical if the subgraph obtained by the removal of an
arbitrary edge admits a cycle-continuous mapping onto the graph
consisting of two vertices joined by three parallel edges. The
latter condition easily translates to the one requiring the
existence of a nowhere-zero $\mathbb{Z}_2\times
\mathbb{Z}_2$-flow on each edge-deleted subgraph, which in turn
implies that critical snarks in the sense of DeVos et al. and
\v S\' amal \cite{dVNR, S-full} coincide with $4$-edge-critical
snarks of da Silva and Lucchesi \cite{dSL}. The same family of
snarks, under the name of $4$-flow-critical snarks, occurs in a
recent survey of edge-uncolourability measures by Fiol et al.
\cite[Section~4.1]{FMS} without any reference to previous work.

To sum up, during the past twenty years critical snarks were
rediscovered several times in one form or another, and within
different contexts. Although partial results concerning the
relationship between the different definitions exist, see
Freitas et al. \cite[Theorem~3.1]{dFdSL} and Fiol et al.
\cite[Theorem~4.5]{FMS}, it has not been fully realised that
all of them actually coincide with either critical or
bicritical snarks of Nedela and \v Skoviera \cite{NS}. In the
present paper we therefore establish this fact in full
generality and explain the relations between various versions
of criticality in detail. Instead of just proving that critical
snarks coincide with $4$-edge-critical snarks and bicritical
snarks coincide with $4$-vertex-critical snarks we investigate
the corresponding reduction operations locally on the pairs of
vertices and show that different operations have the same
effect. The advantage of this approach becomes evident also in
the fact that these operations can be applied to strong snarks
introduced by Jaeger \cite{Jaeger:CDC, Jaeger:5-col} (see also
Brinkmann et al. \cite[Section~4.6]{BGHM}), which gives rise to
a necessary and sufficient condition for a snark to be strong.

\section{Definitions and preliminaries}\label{sec:def}

In this section we fix the terminology for the rest of this
paper. Our graphs may have parallel edges and loops.
Occasionally we also allow dangling edges. We assume the basic
knowledge of edge-colourings and nowhere-zero flows on graphs.
For more details we recommend the reader to consult Diestel
\cite{D}.

As mentioned above, we define a \textit{snark} to be a
connected cubic graph that does not admit a proper
edge-colouring with three colours; equivalently, a snark is a
connected cubic graph that does not admit a nowhere-zero
$4$-flow. The smallest snark is the \textit{dumbbell graph},
which has two vertices joined by an edge and a loop attached to
each vertex. The smallest bridgeless snark is, of course, the
Petersen graph.

We now introduce the operations related to critical and
flow-critical snarks. Given a graph $G$ and an edge $e$ of $G$,
we let $G-e$ denote the subgraph of $G$ obtained by the removal
of $e$, and $G\sim e$ the cubic graph which arises from $G-e$
by suppressing the resulting two vertices of degree two. By
$G/e$ we denote the graph obtained from $G$ by the contraction
of $e$.

Let $u$ and $v$ be two distinct vertices of $G$. By $G-\{u,v\}$
we denote a graph created from $G$ by removing $u$ and $v$ but
retaining the dangling edges. Note that this deviation from the
standard meaning of the vertex removal has no effect on the
existence of a $3$-edge-colouring but is important for the
existence of nowhere-zero flows. By $G/\{u,v\}$ we denote the
graph obtained from $G$ by identifying $u$ and $v$. If $u$ and
$v$ are connected by an edge~$e$, then $G/e$ arises from
$G/\{u,v\}$ by removing the loop resulting from~$e$.

We proceed to the central concepts of this paper. Following
Nedela and \v Skoviera \cite{NS} we define a snark $G$ to be
\textit{critical} if $G-\{u,v\}$ is $3$-edge-colourable for
every pair of adjacent vertices $u$ and $v$, and
\textit{bicritical} if $G-\{u,v\}$ is $3$-edge-colourable for
every pair of distinct vertices $u$ and $v$. Thus every
bicritical snark is critical, but not necessarily vice versa.

Let $G$ be an arbitrary graph, not necessarily cubic, and let
$k\ge 2$ be an integer. Following da Silva and Lucchesi
\cite{dSL} (see also Carneiro et al. \cite{CdSM}) we say that
$G$ is \textit{$k$-edge-critical} if $G$ does not admit a
nowhere-zero $k$-flow but for each edge $e$ the graph $G/e$
does. We further say that $G$ is \textit{$k$-vertex-critical}
if it does not admit a nowhere-zero $k$-flow but for any two
distinct vertices $u$ and $v$ the graph $G/\{u,v\}$ does. We
now apply these definitions to snarks with $k=4$. Taking into
account the fact that the presence of a loop at a vertex has no
effect on the existence of a nowhere-zero $k$-flow, we can
define a snark to be \textit{$4$-edge-critical} if $G/\{u,v\}$
has a nowhere-zero $4$-flow for any two adjacent vertices, and
\textit{$4$-vertex-critical} if $G/\{u,v\}$ has a nowhere-zero
$4$-flow for any two distinct vertices. Although these snarks
can be encountered in the literature under different names, we
have decided to adopt the terminology used in \cite{CdSM, dSL}
and also in the snark section of the database ``House of
Graphs'' \cite{BCGM:House}.

Before proceeding further we need to mention an important
connection of critical snarks to irreducible snarks, which were
introduced in \cite{NS} and thoroughly studied in~\cite{ChS}. A
snark $G$ is said to be \textit{$k$-irreducible} for a given
integer $k\ge 1$ if removing fewer than $k$ edges from $G$ does
not produce a component with chromatic index $4$ which could be
completed to a cubic graph $H$ of order smaller than $G$. The
resulting graph $H$ is a snark and is called a
\textit{$k$-reduction} $G$. A snark is called
\textit{irreducible} if it is $k$-irreducible for every $k\ge
1$. The following result was proved in \cite{NS}.

\begin{theorem}\label{thm:irred}
The following statements are true for an arbitrary snark $G$.
\begin{itemize}
\item[{\rm (i)}] If $1\leq k\leq 4$, then $G$ is
    $k$-irreducible if and only if $G$ is either cyclically
    $k$-connected or the dumbbell graph.
\item[{\rm (ii)}] If $k\in\{5,6\}$, then $G$ is
    $k$-irreducible if and only if it is critical.
\item[{\rm (iii)}] If $k\geq 7$, then $G$ is
    $k$-irreducible if and only if it is bicritical.
\end{itemize}
\end{theorem}

As a direct consequence of Theorem~\ref{thm:irred} we obtain
the fact that a snark is irreducible if and only if it is
bicritical. Furthermore, in a bicritical snark the removal of
every nontrivial edge-cut (one that is different from three
edges incident with a vertex) produces only colourable
components. The following corollary is also important.

\begin{corollary}
Every critical snark is cyclically $4$-edge-connected and has
girth at least $5$.
\end{corollary}

\section{Critical and flow-critical snarks}

We start by exploring the effect of various operations
occurring in the definitions of critical and flow-critical
snarks.

\begin{theorem}\label{thm:local}
Let $G$ be a snark and let $u$ and $v$ be two distinct vertices
of $G$. The following statements {\rm (i)-(iii)} are
equivalent. If, in addition, $u$ and $v$ are adjacent and
joined by an edge $e$, then all the statements  {\rm (i)-(vi)}
are equivalent.
\begin{itemize}
 \item[{\rm (i)}] $G-\{u,v\}$ is $3$-edge-colourable.

 \item[{\rm (ii)}] $G-\{u,v\}$ admits a nowhere-zero
     $4$-flow.

 \item[{\rm (iii)}] $G/\{u,v\}$ admits a nowhere-zero
     $4$-flow.

\item[{\rm (iv)}] $G-e$ admits a nowhere-zero $4$-flow.

\item[{\rm (v)}] $G/e$ admits a nowhere-zero $4$-flow.

\item[{\rm (vi)}] $G\sim e$ is $3$-edge-colourable.
\end{itemize}
\end{theorem}

\begin{proof}
(i) $\Rightarrow$ (ii): Assume that the graph $G-\{u,v\}$ is
$3$-edge-colourable. If the colours for this colouring are
taken to be the non-zero elements of the group
$\mathbb{Z}_2\times \mathbb{Z}_2$, then the colouring is at the
same time a nowhere-zero $\mathbb{Z}_2\times \mathbb{Z}_2$-flow
on $G-\{u,v\}$. Tutte's equivalence theorem (see
\cite[Theorem~6.3.3 and Corollary 6.3.2]{D}) now implies that
$G-\{u,v\}$ admits a nowhere-zero $4$-flow.

\medskip

(ii) $\Rightarrow$ (iii): Assume that $G-\{u,v\}$ admits a
nowhere-zero $4$-flow. Without loss of generality we may assume
that the underlying orientation has all the dangling edges of
$G-\{u,v\}$ directed outward. Kirchhoff's law now implies that
the sum of values on the dangling edges is $0$. Since
$G/\{u,v\}$ arises from $G-\{u,v\}$ by attaching the dangling
edges to a new vertex, we infer that the induced valuation is a
nowhere-zero $4$-flow on $G/\{u,v\}$.

\medskip

(iii) $\Rightarrow$ (i): Assume that $G/\{u,v\}$ admits a
nowhere-zero $4$-flow. By Tutte's equivalence theorem it also
admits a nowhere-zero $\mathbb{Z}_2\times \mathbb{Z}_2$-flow.
If we regard the flow values in $\mathbb{Z}_2\times
\mathbb{Z}_2$ as colours and remove the vertex $u=v$ from
$G/\{u,v\}$, we immediately obtain a 3-edge-colouring of
$G-\{u,v\}$.

\medskip

For the rest of the proof we assume that the vertices $u$ and
$v$ are joined by an edge $e$.

\medskip

(i) $\Rightarrow$ (vi): Assume that $G-\{u,v\}$ is $3$-edge
colourable. Since $G$ is a snark, every $3$-edge-colouring of
$G-\{u,v\}$ must assign the same colour to both dangling edges
formerly incident with one of $u$ and $v$. A $3$-edge-colouring
of $G-\{u,v\}$ is at the same time a nowhere-zero
$\mathbb{Z}_2\times \mathbb{Z}_2$-flow, so Kirchhoff's law
implies that the dangling edges incident with the other vertex
also have the same colour. It follows that every
$3$-edge-colouring of $G-\{u,v\}$ induces a $3$-edge-colouring
of $G\sim e$.

\medskip

(vi) $\Rightarrow$ (iv): Assume that $G\sim e$ is
$3$-edge-colourable. Then $G\sim e$ has a nowhere-zero
$\mathbb{Z}_2\times \mathbb{Z}_2$-flow and hence a nowhere-zero
$4$-flow. Since $G-e$ is a subdivision of $G\sim e$, it follows
that $G-e$ has a nowhere-zero $4$-flow as well.

\medskip

(iv) $\Rightarrow$ (v): If $G-e$ has a nowhere-zero $4$-flow,
then so does the graph obtained from $G-e$ by identifying $u$
and $v$, which is exactly $G/e$.

\medskip

(v) $\Rightarrow$ (i): If $G/e$ has a nowhere-zero $4$-flow,
then so does the graph obtained from $G/e$ by removing the
vertex corresponding to $e$ and by retaining the dangling
edges. Obviously, the latter graph is isomorphic to
$G-\{u,v\}$. By Tutte's equivalence theorem, $G-\{u,v\}$ has a
nowhere-zero $\mathbb{Z}_2\times \mathbb{Z}_2$-flow, which at
the same time is a $3$-edge-colouring of $G-\{u,v\}$.
\end{proof}

\medskip

\noindent\textbf{Remark.} The equivalence between statements
(i) and (vi) of Theorem~\ref{thm:local} was first observed in
\cite[Proposition~4.2]{NS}.

\bigskip

The next two theorems are immediate consequences of
Theorem~\ref{thm:local} and Theorem~\ref{thm:irred}.

\begin{theorem}\label{thm:critical}
The following statements are equivalent for an arbitrary snark
$G$.
\begin{itemize}
\item[{\rm (i)}] $G$ is critical.

\item[{\rm (ii)}] $G$ is $4$-edge-critical.

\item[{\rm (iii)}] $G$ is $5$-irreducible.

\item[{\rm (iv)}] $G$ is $6$-irreducible.

\item[{\rm (v)}] $G\sim e$ is $3$-edge-colourable for each
    edge $e$ of $G$.
\end{itemize}
\end{theorem}

The equivalence (i) $\Leftrightarrow$ (ii) in
Theorem~\ref{thm:critical} explains that the property of being
a $4$-edge-critical snark is the same as being critical. This
fact has been recently observed by de Freitas et al.
\cite[Theorem~3.1]{dFdSL} and independently by Fiol et al.
\cite[Theorem~4.5]{FMS}. Surprisingly, an analogous statement
for $4$-edge-critical and bicritical snarks has so far escaped
attention. We formulate this fact in the next theorem, thereby
completing the relationship between critical and flow-critical
snarks.

\begin{theorem}\label{thm:bicritical}
The following statements are equivalent for an arbitrary snark
$G$.
\begin{itemize}
\item[{\rm (i)}] $G$ is bicritical.

\item[{\rm (ii)}] $G$ is $4$-vertex-critical.

\item[{\rm (iii)}] $G$ is $7$-irreducible.

\item[{\rm (iv)}] $G$ is irreducible.
\end{itemize}
\end{theorem}

It is worth mentioning that there exist \textit{strictly
critical} snarks -- snarks that are critical but not
bicritical. The first known examples were found independently
by Chladn\'y in his Master Thesis (see \cite{ChS}) and by
Steffen \cite{Steffen:strict}. Somewhat later, Gr\"unewald and
Steffen \cite{GS} presented a construction of cyclically
$5$-edge-connected strictly critical snarks. Strictly critical
snarks whose cyclic connectivity equals $4$ were completely
characterised by Chladn\'y and \v Skoviera
\cite[Section~6]{ChS}, providing a deeper insight into what
makes snarks strictly critical. They also showed that there
exist strictly critical snarks of order $n$ for every even
integer $n\ge 32$. On the other hand, an exhaustive computer
search performed by Brinkmann and Steffen \cite{BS} revealed
that there are no strictly critical snarks of any order smaller
than $32$.

Strictly critical snarks have resurfaced within the
flow-critical context in a recent work of Carneiro et al.
\cite{CdSM}. They devised an exponential-time algorithm that
verifies whether a snark is $4$-edge-critical or
$4$-vertex-critical, and applied the algorithm to the body of
all cyclically $4$-edge-connected snarks of order at most $36$
with girth at least $5$ generated by Brinkmann et al.
\cite{BGHM}. The use of this algorithm allowed them to compile
complete lists of critical, bicritical, and strictly critical
snarks of every order not exceeding $36$. The lists are
available in the snark section of the database ```House of
Graphs'' \cite{BCGM:House}. It transpires that among all snarks
of order at most $36$ there are exactly $55172$ critical
snarks, but only $846$ of them are strictly critical, just
slightly over $1.5$ percent. We have verified that all of them
have cyclic connectivity $4$. (The number $837$ of strictly
critical snarks of order not exceeding $36$ mentioned in
\cite[Section~3]{CdSM} is incorrect.)

Theorems~\ref{thm:critical} and~\ref{thm:bicritical} suggest
that the algorithm of Carneiro et al. \cite{CdSM} to check
flow-criticality of a given snark $G$ can be simplified if we
consider criticality instead. Indeed, the algorithm for
flow-criticality fixes an orientation of $G$ and for a chosen
pair $(u,v)$ of vertices it attempts to construct a weight
function with values in $\mathbb{Z}_4$ under which the
Kirchhoff law fails only at $u$ and $v$. This requires, in
particular, checking several possibilities for balanced weight
assignments at each vertex $w$ different from $u$ and $v$. Our
Theorems~\ref{thm:critical} and~\ref{thm:bicritical} imply that
if we verify criticality instead, then no orientation is
required, and for each vertex there is only one possibility for
a balanced assignment, up to a permutation of colours. This
approach might prove useful in testing irreducibility of large
individual snarks or large sets of snarks.

\section{Application to strong snarks}

We finish this paper by applying Theorem~\ref{thm:local} to
strong snarks introduced by Jaeger in \cite{Jaeger:CDC,
Jaeger:5-col}. Using the notation introduced in
Section~\ref{sec:def} we recall that a snark $G$ is said to be
\textit{strong} if $G\sim e$ is a snark for each edge $e$.

Our final result is a direct consequence of the equivalence (i)
$\Leftrightarrow$ (vi) from Theorem~\ref{thm:local}.

\begin{theorem}\label{thm:strong}
A snark $G$ is strong if and only if $G-\{u,v\}$ has chromatic
index $4$ for every pair of adjacent vertices $u$ and $v$.
\end{theorem}

At the first glance Theorem~\ref{thm:strong} might seem to
suggests that strong snarks are in some sense similar to
critical snarks except that they lie on the other side of the
colourability spectrum. Unfortunately, this is not the case.
Indeed, there is no reason for strong snarks to be nontrivial
in the usual sense, that is, to be cyclically
$4$-edge-connected and have girth at least $5$. For example,
replacing any vertex of a strong snark with a triangle produces
another strong snark.

\subsection*{Acknowledgments}
The first author was partially supported by VEGA 1/0876/16. The
second author was partially supported by APVV-15-0220.

\end{document}